\documentclass[11pt,amssymb]{amsart}
\usepackage{amssymb}
\usepackage[mathscr]{eucal}
\usepackage[all,cmtip]{xy}

\newcommand{\m}{\mathfrak{m}}

\newcommand{\Z}{{\mathbb Z}}
\newcommand{\N}{{\mathbb N}}

\newcommand{\C}{{\mathbb C}}
\newcommand{\Q}{{\mathbb Q}}

\newcommand{\R}{{\mathbb R}}

\newcommand{\G}{\mathcal{G}}
\newtheorem{thm}{Theorem}[section]
\newtheorem{lemma}[thm]{Lemma}
\newtheorem{prop}[thm]{Proposition}
\newtheorem{cor}[thm]{Corollary}

\newcommand{\Hom}{\mathrm{Hom}}

\newcommand{\St}{\mathrm{St}}
\begin{document}

\title[Character varieties]{On the character varieties of finitely generated groups}

\begin{abstract}
We establish three results dealing with the character varieties of finitely generated groups. The first two are concerned with the behavior of $\varkappa(\Gamma, n) = \dim X_n (\Gamma)$ as a function of $n$, and the third addresses the problem of realizing a $\Q$-defined complex affine algebraic variety as a character variety.
\end{abstract}

\author[I.A.~Rapinchuk]{Igor A. Rapinchuk}

\address{Department of Mathematics, Harvard University, Cambridge, MA 02138}

\email{rapinch@math.harvard.edu}

\maketitle

\section{Introduction}\label{S:I}

Let $\Gamma$ be a finitely generated group. It is well-known (cf., for example, \cite{LM}) that for each integer $n \geq 1,$ there exists a $\Q$-defined affine algebraic variety $R_n (\Gamma)$, called the {\it $n$-th representation variety of} $\Gamma$, such that for any field $F$ of characteristic 0, there is a natural bijection between the representations $\rho \colon \Gamma \to GL_n (F)$ and the set $R_n (\Gamma, F)$ of $F$-rational points of $R_n (\Gamma)$. (In this paper, we will fix an algebraically closed field $K$ of characteristic 0 and identify $R_n (\Gamma)$ with $R_n (\Gamma, K)$). Furthermore, there is a natural (adjoint) action of $GL_n$ on $R_n (\Gamma)$. The corresponding categorical quotient $X_n (\Gamma)$ is a $\Q$-defined affine algebraic variety called the {\it $n$-th character variety of} $\Gamma$; its points parametrize the isomorphism classes of completely reducible representations of $\Gamma$.

The goal of this paper is to establish three results about the character varieties of finitely generated groups. The first two deal with the dimension $\varkappa (\Gamma, n) := \dim X_n (\Gamma)$ as a function of $n.$ It is easy to see that $\varkappa (\Gamma, n) \leq \varkappa (\Gamma, m)$ for $n \leq m$ (cf. Lemma \ref{L:CharDim}), so the function $\varkappa (\Gamma, n)$ is increasing, and one would like to understand its rate of growth. In the case where $\Gamma = F_d$, the free group on $d$ generators, we have
$$
\varkappa (F_d, n) = \left\{ \begin{array}{cc} n, & d = 1 \\ (d-1)n^2 +1, & d > 1 \end{array} \right.
$$
It follows that the growth of $\varkappa(\Gamma, n)$ is {\it at most quadratic in $n$} for any $\Gamma.$ At the other end of the spectrum are the groups for which $\varkappa (\Gamma, n) = 0$ for all $n$ -- such groups are called {\it $SS$-rigid}. For example, according to Margulis's Superrigidity Theorem \cite[Ch. VII]{Ma}, irreducible lattices in higher rank Lie groups are $SS$-rigid. The question is {\it what rate of growth of $\varkappa (\Gamma, n)$ as a function of $n$ can actually occur for a finitely generated group $\Gamma$?} Our first result shows that if $\Gamma$ is \emph{not} $SS$-rigid, then the growth of $\varkappa(\Gamma, n)$ is {\it at least linear.}

\vskip2mm

\noindent {\bf Theorem 1.} {\it Let $\Gamma$ be a finitely generated group. If $\Gamma$ is not $SS$-rigid, then there exists a linear function $f (t) = at + b$ with $a > 0$  such that
$$
\varkappa(\Gamma, n) := \dim X_n (\Gamma) \geq f(n)
$$
for all $n \geq 1$.}

\vskip2mm
\noindent The proof will be given in \S \ref{S:SS}.

Our second result provides a large family of groups for which the rate of growth of $\varkappa(\Gamma, n)$ is indeed linear. The set-up is as follows. Let $\Phi$ be a reduced irreducible root system of rank $\geq 2$, $G$ the corresponding universal Chevalley-Demazure group scheme, and $R$ a finitely generated commutative ring. Then, it is known that the elementary subgroup $G(R)^+ \subset G(R)$ (i.e. the subgroup generated by the $R$-points of the 1-parameter root subgroups) has Kazhdan's property (T) (see \cite{EJK}), and hence is finitely generated. We note that in \cite{RS}, the finite presentation of Steinberg groups (which clearly implies the finite generation of $G(R)^+$) is proved directly in the case where rank $\Phi \geq 3$ (see also \cite{R} for a discussion of the finite generation of Chevalley groups of rank $\geq 2$ over the polynomial ring $k[t],$ with $k$ a field). Assume furthermore that $(\Phi, R)$ is a {\it nice pair}, that is $2 \in R^{\times}$ if $\Phi$ contains a subsystem of type $\mathsf{B}_2$ and $2,3 \in R^{\times}$ if $\Phi$ is of type $\mathsf{G}_2.$ In this paper, we give an alternative proof of the following result that we first established in \cite[Theorem 2]{IR1}.


\vskip2mm
\noindent {\bf Theorem 2.} {\it Let $\Phi$ be a reduced irreducible root system of rank $\geq 2$, $R$ a finitely generated commutative ring such that $(\Phi, R)$ is a nice pair, and $G$ the universal Chevalley-Demazure group scheme of type $\Phi$. Let $\Gamma = G(R)^+$ denote the elementary subgroup of $G(R)$ and consider the $n$-th character variety $X_n (\Gamma)$ of $\Gamma$ over an algebraically closed field $K$ of characteristic 0. Then there exists a constant $c = c(R)$ (depending only on $R$) such that $\varkappa(\Gamma, n) := \dim X_n (\Gamma)$ satisfies
\begin{equation}\label{E:LG}
\varkappa(\Gamma, n) \leq c \cdot n
\end{equation}
for all $n \geq 1.$}

\vskip2mm

We should add that since the elementary subgroups of Chevalley groups over finitely generated rings constitute essentially all known examples of discrete linear groups having Kazhdan's property (T), we were led in \cite{IR1} to formulate the following.

\vskip2mm

\noindent {\bf Conjecture.} {\it Let $\Gamma$ be a discrete linear group having Kazhdan's property (T). Then there exists a constant $c = c(\Gamma)$ such that $$\varkappa(\Gamma, n) \leq c \cdot n$$ for all $n \geq 1.$}

\vskip2mm

A question that remains is whether there exist groups $\Gamma$ for which the growth of $\varkappa (\Gamma, n)$ is strictly between linear and quadratic, i.e. we have
$$
\limsup_{n \to \infty} \frac{\varkappa(\Gamma, n)}{n} = \infty \ \ \ \text{and} \ \ \ \liminf_{n \to \infty} \frac{\varkappa(\Gamma, n)}{n^2} = 0.
$$
At present, no examples are known, though a possible approach to constructing such groups, which relies on the analysis of certain proalgebraic completions, has been suggested by M.~Kassabov.

The proof of Theorem 2 that we gave in \cite{IR1}
was based on the idea, going back to A.~Weil, of bounding the dimension of the tangent space to $X_n (\Gamma)$ at a point $[\rho]$ corresponding to a representation $\rho \colon \Gamma \to GL_n (K)$ by the dimension of the cohomology group $H^1 (\Gamma, \mathrm{Ad} \circ \rho).$ Then, using our rigidity results from \cite{IR}, we were able to relate the latter space to a certain space of derivations of $R$. In \S \ref{S:LG}, we give a new argument which is more geometric in nature and depends on a refined version of our previous rigidity statements
(see Theorem \ref{T:Rigidity}).

Finally, we would like to address the question of which $\Q$-defined complex algebraic varieties can actually occur as $X_n (\Gamma)$, for some finitely generated group $\Gamma$ and integer $n.$
This question was initially studied by M.~Kapovich and J.~Millson \cite{KM} in connection with their work on Serre's problem of determining which finitely presented groups occur as fundamental groups of smooth complex algebraic varieties. They proved the following result, which basically provides an answer up to birational isomorphism, and was a crucial ingredient in their construction of Artin and Shephard groups that are not fundamental groups of smooth complex algebraic varieties.

\vskip2mm

\noindent {\bf Theorem.} (\cite{KM}, Theorem 1.2) {\it For any affine variety $S$ defined  over
$\Q$, there is  an Artin group $\Gamma$ such that a
Zariski-open subset $U$ of $S$ is biregular isomorphic to  a
Zariski-open subset of $X (\Gamma, PO(3)).$}

\vskip2mm

We have been able to realize an arbitrary affine algebraic variety as a character variety ``almost" up to biregular isomorphism.


\vskip2mm

\noindent {\bf Theorem 3.} {\it Let $S$ be an affine algebraic variety defined over $\mathbb{Q}.$ There exist a finitely generated group $\Gamma$ having Kazhdan's property $(T)$ and an integer $n \geq 1$ such that there is a biregular $\Q$-defined isomorphism of complex algebraic varieties
$$
S (\mathbb{C}) \to X_n (\Gamma) \setminus \{ [\rho_0] \},
$$
where $\rho_0$ is the trivial representation and $[\rho_0]$ is the corresponding point of $X_n (\Gamma)$.}

\vskip2mm

\noindent (We note that for a Kazhdan group, $[\rho_0]$ is always an isolated point (cf. \cite[Proposition 1]{AR}).)

We would like to point out that the proofs of Theorems 2 and 3 rely extensively on our rigidity results contained in \cite{IR} and \cite{IR1}.

\vskip2mm

\noindent {\bf Acknowledgments.} I would like to thank Professor Gregory A. Margulis for his continued interest in my work. I would also like to thank Martin Kassabov for communicating to me some ideas that may lead to a construction of finitely generated groups for which the growth of $\varkappa(\Gamma, n)$ is strictly between linear and quadratic. The author was partially supported by an NSF Postdoctoral Fellowship.

\section{Preliminaries on character varieties}\label{S:Preliminaries}

In this section, we summarize some results on character varieties that will be needed later.
Throughout this section, we will work over a fixed algebraically closed field $K$ of characteristic 0.

We begin by recalling a couple of statements from \cite{AR1} concerning the relationship between the character varieties of a finitely generated group and those of its finite-index subgroups. First, suppose $f \colon Y \to Z$ is a morphism of affine algebraic varieties. We will say that $f$ is {\it quasi-finite} if it has finite fibers and {\it integral} if the induced map on coordinate rings $f^* \colon K[Z] \to K[Y]$ is integral\footnotemark \footnotetext{Recall that a ring homomorphism $f \colon A \to B$ is said to be {\it integral} if $B$ is integral over $f(A)$.} (clearly, an integral morphism is quasi-finite).  Now let $\Gamma$ be any finitely generated group and $\Delta \subset \Gamma$ be a finite-index subgroup. For any $n \geq 1$, restricting representations yields a regular map
\begin{equation}\label{E:RestrictionRep}
\mathrm{Res}_{\Delta}^{\Gamma} \colon R_n (\Gamma) \to R_n (\Delta).
\end{equation}
Since $\mathrm{Res}_{\Delta}^{\Gamma}$ clearly commutes with the adjoint action of $GL_n$, it induces a regular map
\begin{equation}\label{E:RestrictionChar}
\nu \colon X_n (\Gamma) \to X_n (\Delta).
\end{equation}

\begin{lemma}\label{L:Restriction}{\rm (\cite{AR1}, Lemma 1)}
For any finite-index subgroup $\Delta \subset \Gamma$, the restriction morphism $$\nu \colon X_n (\Gamma) \to X_n (\Delta)$$ is integral (and hence quasi-finite).
\end{lemma}

Suppose now that $\Delta \subset \Gamma$ has index $d.$ If we fix a system of representatives $\gamma_1, \dots, \gamma_d$ for $\Gamma/ \Delta,$ then we recall that for any representation $\rho \in R_m (\Delta),$ the {\it induced representation} $\tilde{\rho} \in R_{md} (\Gamma)$ has the following description: for $\gamma \in \Gamma$, the matrix $\tilde{\rho} (\gamma)$ consists of $d \times d$ blocks of size $m \times m$, the $ij$-th block being nonzero only if $\gamma \gamma_j = \gamma_i \delta$ with $\delta \in \Delta$, in which case it equals $\rho (\delta).$ It follows that the correspondence $\rho \mapsto \tilde{\rho}$ defines a regular map
$$
\alpha = \mathrm{Ind}_{\Delta}^{\Gamma} \colon R_m (\Delta) \to R_{md} (\Gamma)
$$
(which depends on the choice of a system of representatives $\Gamma / \Delta$). Since $\alpha$ is obviously compatible with the adjoint actions of $GL_m$ and $GL_{md}$ on the respective representation varieties, it descends to a morphism $\bar{\alpha} \colon X_m (\Delta) \to X_{md} (\Gamma)$ of the corresponding character varieties (which is independent of the choice of coset representatives).

\begin{lemma}\label{L:Induction}{\rm (\cite{AR1}, Lemma 3)}
If $\Delta \subset \Gamma$ is a subgroup of index $d$, then the induction morphism $$\bar{\alpha} \colon X_m (\Delta) \to X_{md} (\Gamma)$$ is quasi-finite.
\end{lemma}

From the lemma, we obtain the following corollary, which will be needed in the proof of Theorem 1.

\begin{cor}\label{C:Induction}
Let $\Delta \subset \Gamma$ be a subgroup of finite index $d$. Then for any $m \geq 1$, we have
$$
\dim X_{md} (\Gamma) \geq \dim X_m (\Delta).
$$
\end{cor}
\begin{proof}
Since the map $\bar{\alpha} \colon X_m (\Delta) \to X_{md} (\Gamma)$ has finite fibers, the statement follows directly from the theorem on the dimension of fibers \cite[\S 6.3, Theorem 7]{Shaf}.
\end{proof}

Next, for any finitely generated group $\Gamma$, let $K[R_n (\Gamma)]$ and $K[X_n (\Gamma)]$ denote the coordinate rings of $R_n (\Gamma)$ and $X_n (\Gamma)$, respectively. It is well-known that $K[X_n (\Gamma)]$ is the subalgebra of $K[R_n (\Gamma)]$ generated by the Fricke functions $\tau_{\gamma} (\rho) = \mathrm{tr} (\rho (\gamma))$ and the inverse of the determinant functions $\delta_{\gamma} (\rho) = \det(\rho (\gamma))^{-1}$ for all $\gamma \in \Gamma$ (see \cite[1.31]{LM}). In fact, since $K[X_n (\Gamma)]$ is a finitely generated algebra, we only need to take the Fricke functions $\tau_{\gamma_1}, \dots, \tau_{\gamma_{\ell}}$ for some {\it finite} set $\{ \gamma_1, \dots, \gamma_{\ell} \} \subset \Gamma$ (and then, by a theorem of Procesi \cite{P}, it follows that the functions $\delta_{\gamma}$ are polynomials in $\tau_{\gamma_1}, \dots, \tau_{\gamma_{\ell}}$).

Suppose now that the Fricke functions $\tau_{\gamma_1}, \dots, \tau_{\gamma_{\ell}}$ are algebraically independent on some $R_n (\Gamma).$ Then, the Fricke functions corresponding to the same elements $\gamma_1, \dots, \gamma_{\ell}$ remain algebraically independent on $R_{m} (\Gamma)$ for any $m > n$ (this follows by considering representations of the form $\rho' = \rho \oplus {\bf 1}_{m - n}$, where $\rho \in R_n (\Gamma)$ and ${\bf 1}_{m - n}$ is the trivial representation of $\Gamma$ of dimension $m - n$). This leads to the following (elementary) statement, which we record for future use.

\begin{lemma}\label{L:CharDim}
For any $m > n,$ we have $\dim X_m (\Gamma) \geq \dim X_n (\Gamma).$
\end{lemma}

\noindent (In fact, it is easy to see that the embedding $R_n (\Gamma) \to R_m (\Gamma)$ that sends $\rho$ to $\rho'$ in the above notations induces an injection $X_n (\Gamma) \to X_m (\Gamma).$)

\vskip2mm

Finally, let us recall the following condition on a group $\Gamma$:

\vskip2mm

\noindent (FAb) \ \ \parbox[t]{15cm}{{\it for any finite-index subgroup $\Delta \subset \Gamma,$ the abelianization $\Delta^{\mathrm{ab}} = \Delta / [\Delta, \Delta]$ is finite}}

\vskip2mm
For groups with this property, we have the following.
\begin{prop}\label{D:P-KG}{\rm (\cite{AR}, Proposition 2)}
Let $\Gamma$ be a group satisfying {\rm (FAb)}. For any $n \geq 1$, there exists a finite collection $G_1, \dots, G_d$ of algebraic subgroups of $GL_n (K)$, such that for any completely reducible representation $\rho \colon \Gamma \to GL_n (K)$, the Zariski closure $\overline{\rho (\Gamma)}$ is conjugate to one of the $G_i.$ Moreover, for each $i$, the connected component $G_i^{\circ}$ is a semisimple group.
\end{prop}

\section{A refined form of rigidity for Chevalley groups}

This section is devoted to establishing some refinements of our rigidity results from \cite{IR} that will be needed in the proofs of Theorems 2 and 3. Our set-up is as follows.
Let $\Phi$ be a reduced irreducible root system of rank $\geq 2$ and $G$ be the corresponding universal Chevalley-Demazure group scheme. For a commutative ring $R$, we let $\Gamma = G(R)^+ \subset G(R)$ be the {\it elementary subgroup}, i.e. the subgroup generated by the $R$-points of the 1-parameter root subgroups.
Throughout this section, we will always assume that $(\Phi, R)$ is a {\it nice pair}, i.e. $2 \in R^{\times}$ if $\Phi$ contains a subsystem of type $\mathsf{B}_2$ and $2, 3 \in R^{\times}$ if $\Phi$ is of type $\mathsf{G}_2.$


\begin{thm}\label{T:Rigidity}
Let $\Phi$ be a reduced irreducible root system of rank $\geq 2$ and $G$ be the corresponding universal Chevalley-Demazure group scheme. Suppose $R$ is a finitely generated commutative ring such that $(\Phi, R)$ is a nice pair, set $\Gamma = G(R)^+$, and let $K$ be an algebraically closed field of characteristic 0. Fix an integer $n \geq 1.$

\vskip1mm

\noindent {\rm (i)} \parbox[t]{16cm}{There exists a finite index subgroup $\Delta \subset \Gamma = G(R)^+$ such that for \emph{any} (nontrivial) completely reducible representation $\rho \colon \Gamma \to GL_n (K)$, we have
$$
\rho \vert_{\Delta} = (\sigma \circ F) \vert_{\Delta},
$$
where
$$F \colon G(R) \to \underbrace{G(K) \times \cdots \times G(K)}_{k \ {\rm copies}}$$
is a group homomorphism arising from a ring homomorphism $$f \colon R \to \underbrace{K \times \cdots \times K}_{k \ {\rm copies}}$$ with Zariski-dense image, $k \leq n$, and $\sigma \colon G(K) \times \cdots \times G(K) \to GL_n (K)$ is a morphism of algebraic groups.}

\vskip1mm

\noindent {\rm (ii)} \parbox[t]{16cm}{There exists an integer $M$ (depending \emph{only} on $n$) such that if $M \in R^{\times}$, then one can take $\Delta = \Gamma.$}

\end{thm}

\vskip2mm

\noindent {\bf Remark 3.2.} We would like to point out the difference between \cite[Main Theorem]{IR} and Theorem \ref{T:Rigidity}: given a representation $\rho \colon \Gamma \to GL_n (K),$ the former guarantees the existence of a finite index subgroup $\tilde{\Delta} \subset \Gamma$ such that $\rho \vert_{\tilde{\Delta}}$ has a standard description. On the other hand, under the additional assumption that $R$ is a finitely generated ring, Theorem \ref{T:Rigidity} gives us a finite index subgroup $\Delta \subset \Gamma$ that works uniformly for \emph{all} completely reducible representations of $\Gamma$ of a given dimension.

\vskip2mm

The theorem will follow from \cite[Main Theorem]{IR} and Lemma \ref{L:Bound} below. Before giving the statement of the lemma, let us recall that to any representation $\rho \colon \Gamma \to GL_n (K)$, one can associate a commutative algebraic ring $A(\rho)$, together with a homomorphism of abstract rings $f \colon R \to A(\rho)$ having Zariski-dense image such that for every root $\alpha \in \Phi$, there is an injective regular map $\psi_{\alpha} \colon A(\rho) \to H$ into $H := \overline{\rho (G(R)^+)}$ (Zariski closure) satisfying
\begin{equation}\label{E:ARR101}
\rho (e_{\alpha} (t)) = \psi_{\alpha} (f(t)).
\end{equation}
for all $t \in R$ (see \cite[Theorem 3.1]{IR}).






\addtocounter{thm}{1}

\begin{lemma}\label{L:Bound}
Let $\Phi$ be a reduced irreducible root system of rank $\geq 2$ and $G$ be the corresponding universal Chevalley-Demazure group scheme. Suppose $R$ is a finitely generated commutative ring such that $(\Phi, R)$ is a nice pair, set $\Gamma = G(R)^+$, and let $K$ be an algebraically closed field of characteristic 0. Fix $n \geq 1$. Given a representation $\rho \colon \Gamma \to GL_n (K)$, we let $A = A(\rho)$ be the associated algebraic ring and denote by $A(\rho)^{\circ}$ the connected component of $A(\rho)$. Then there exists an integer $N = N(n) \geq 1$ (depending \emph{only} on $n$) such that $[A(\rho) : A(\rho)^{\circ}] \leq N$ for \emph{any} completely reducible representation $\rho \colon \Gamma \to GL_n (K).$
\end{lemma}
\begin{proof}
First, as we remarked earlier, $\Gamma$ has Kazhdan's property $(T)$ \cite{EJK}, and consequently satisfies condition (FAb) (cf. \cite{HV}). So, it follows from Proposition \ref{D:P-KG}
that there exists an integer $m_1 \geq 1$ (depending only on $n$) such that $[H : H^{\circ}] \leq m_1$ for any completely reducible representation $\rho,$ where $H = \overline{\rho(\Gamma)}.$

Now fix some completely reducible representation $\rho \colon \Gamma \to GL_n (K)$ and let $A$ and $f \colon R \to A$
be the associated algebraic ring and ring homomorphism.
Since $K$ has characteristic 0, we have $A = A^{\circ} \oplus C$, where $C$ is a finite ring (cf. \cite[Proposition 2.14]{IR}). By \cite[Proposition 5.3]{IR}, the connected component $H^{\circ}$ coincides with the subgroup of $H$ generated by all of the $\psi_{\alpha}(A^{\circ})$, for $\alpha \in \Phi.$ Let us fix a root $\alpha_0.$ We claim that $\psi_{\alpha_0} (C)$ centralizes $H^{\circ}.$ Indeed, since for any $\alpha, \beta \in \Phi$, $\alpha \neq - \beta,$ we have the commutator formula
$$
[\psi_{\alpha} (a), \psi_{\beta} (b)] = \prod \psi_{i \alpha + j \beta} (N^{i,j}_{\alpha, \beta} a^i b^j)
$$
(see \cite[Proposition 4.2]{IR}), it follows that
$$
[\psi_{\alpha} (a), \psi_{\alpha_0}(b)] = 1
$$
for any $\alpha \neq -\alpha_0$ and $a \in A^{\circ}, b \in C.$ On the other hand, it is easy to show using the fact that the Weyl group is generated by the standard reflections $w_{\alpha}$ for $\alpha \neq \pm \alpha_0$ (see, for example, the proof of \cite[Lemma 4.1]{CongKer}) that
$$
H^{\circ} = \langle \psi_{\alpha} (A^{\circ}) \mid \alpha \in \Phi, \alpha \neq - \alpha_0 \rangle,
$$
which yields our claim.

Next, let $\Theta = Z_H (H^{\circ})$ denote the centralizer of $H^{\circ}$ in $H$. Clearly, we have $\Theta \cap H^{\circ} \subset Z(H^{\circ})$, where $Z(H^{\circ})$ is the center of $H^{\circ}.$ Since there are obviously only finitely many possibilities for $H^{\circ}$ and $H^{\circ}$ is semisimple, there exists an integer $m_2$ (depending only on $n$) such that $\vert \Theta \cap H^{\circ} \vert \leq \vert Z(H^{\circ}) \vert \leq m_2.$ Thus,
$$
\vert \Theta \vert = [ \Theta : \Theta \cap H^{\circ} ]\vert \Theta \cap H^{\circ} \vert \leq m_1 \cdot m_2.
$$
Set $N = m_1 \cdot m_2.$ Since $\psi_{\alpha_0} (C) \subset \Theta$ and $\psi_{\alpha_0}$ is injective, the lemma follows.

\end{proof}

\vskip2mm

\noindent {\it Proof of Theorem \ref{T:Rigidity}.} (i) Let $\rho \colon \Gamma \to GL_n (K)$ be any completely reducible representation. Then, combining \cite[Main Theorem, Proposition 2.20, and Lemma 5.7]{IR} and \cite[Lemma 4.2]{IR1}, we obtain a finite-dimensional commutative $K$-algebra
$$
B = \underbrace{K \times \cdots \times K}_{k \ {\rm copies}}
$$
with $k \leq n$, together with a ring homomorphism $f \colon R \to B$ with Zariski-dense image, and a morphism of algebraic groups $\sigma \colon G(B) \to H = \overline{\rho(\Gamma)}$, such that on a suitable finite-index subgroup $\tilde{\Delta} \subset \Gamma$, we have
\begin{equation}\label{E:StandardDes}
\rho \vert_{\tilde{\Delta}} = (\sigma \circ F) \vert_{\tilde{\Delta}},
\end{equation}
where $F \colon \Gamma \to G(B)$ is the group homomorphism induced by $f.$ By construction, $B$ is the connected component $A^{\circ}(\rho)$ of the algebraic ring $A(\rho)$ associated to $\rho.$
Since char $K = 0,$ we can write
$$
A(\rho) = A(\rho)^{\circ} \oplus C
$$
where $C$ is a finite ring \cite[Proposition 2.14]{IR}, and by Lemma \ref{L:Bound}, there is a uniform bound on the size of $C$. Hence, we have a uniform bound on the size of $G(C)$, or equivalently, on the index of $G(A(\rho)^{\circ})$ in $G(A(\rho))$: let $N' \geq 1$ such that
$$
[G(A(\rho)) : G(A(\rho)^{\circ})] \leq N'
$$
for any completely reducible representation $\rho \colon \Gamma \to GL_n (K).$ Next, it follows from the discussion at the beginning of \cite[\S 5]{IR} that the group $\tilde{\Delta}$ appearing in (\ref{E:StandardDes}) coincides with $F^{-1}(G(A^{\circ}))$, where, as before, $F \colon \Gamma \to G(A(\rho))$ is the group homomorphism induced by the ring homomorphism $f \colon R \to A(\rho).$ Now, it is well-known that a finitely generated group has only finitely many subgroups of a given index, and obviously the intersection of finitely many subgroups of finite index is again a finite index subgroup. Thus, taking $\Delta$ to be the the intersection of all subgroups of $\Gamma$ of index $\leq N'$ 
completes the proof.


\vskip2mm

\noindent (ii) As we already saw in the proof of (i), the need to pass to a finite index subgroup arises from the presence of a finite ring $C$ in the decomposition $A(\rho) = A(\rho)^{\circ} \oplus C.$ Let $M = N!$, where $N$ is the integer appearing in the statement of Lemma \ref{L:Bound}, so that $\vert C \vert \leq N$, and let $f \colon R \to A(\rho)$ be the ring homomorphism with Zariski-dense image associated to $\rho.$ Then, if $M \in R^{\times},$ it follows that $M \in A(\rho)^{\times}$. On the other hand, by construction $M$ annihilates $C$, so $C = 0.$

\hfill $\Box$

\vskip3mm


\noindent {\bf Remark 3.4.} The fact that $R$ is finitely generated was used in the proof of Theorem \ref{T:Rigidity}(ii) in order to conclude that $\Gamma$ satisfies (FAb) and then invoke Proposition \ref{D:P-KG}. While this is the only case we will need to consider in the present paper, we would like to point that the result actually holds without this assumption. However, the argument becomes more involved and we will only sketch it for the sake of completeness. One of the ingredients is the following well-known result of Jordan: {\it Let $K$ be an algebraically closed field of characteristic 0. There exists a function $j \colon \N \to \N$ such that if $\mathcal{G} \subset GL_m (K)$ is a finite group, then $\mathcal{G}$ contains an abelian normal subgroup $\mathcal{N}$ whose index $[\mathcal{G} : \mathcal{N}]$ divides $j(m)$} (we refer the reader to \cite{EB} for a discussion of Jordan's original proof).


As before, we consider a completely reducible representation $\rho \colon \Gamma \to GL_n (K)$. Let $A(\rho)$ and $f \colon R \to A(\rho)$ be the algebraic ring and ring homomorphism associated to $\rho.$ Since char $K = 0$, we have a decomposition
$$
A(\rho) = A(\rho)^{\circ} \oplus C,
$$
and we need to show that $C = 0$ provided that a sufficiently large integer $M$, to be specified below, is invertible in $R$. Now, by \cite[Proposition 4.2]{IR}, there exists a group homomorphism
\begin{equation}\label{E:Steinberglift}
\tilde{\tau} \colon \St(\Phi, A) \to H, \ \ \ \tilde{x}_{\alpha} (a) \mapsto \psi_{\alpha} (a) \ \ \text{for all} \ \ a \in A, \ \alpha \in \Phi,
\end{equation}
where for any commutative ring $S$ and root system $\Phi,$ we let $\St (\Phi, S)$ denote the corresponding Steinberg group.
So, in view of the injectivity of the maps $\psi_{\alpha}$, for all $\alpha \in \Phi$, it suffices to show that there are no nontrivial group homomorphisms $\tilde{\rho} \colon \St (\Phi, C) \to GL_m (K).$



First, since $C$ is finite, hence artinian, we can write it as a product of local rings $$C = C_1 \times \cdots \times C_r.$$ Then $\St (\Phi, R) = \St (\Phi, C_1) \times \cdots \times \St (\Phi, C_r),$  and it suffices to show that there are no nontrivial homomorphisms $\tilde{\rho}_i \colon \St (\Phi, C_i) \to GL_m (K)$ for all $i.$ Thus, we may assume without loss of generality that $C$ is a local ring with maximal ideal $\m$ and residue field $C/ \m = \mathbb{F}_q.$ Recall that we have $G(S)^+ = G(S)$ for any semilocal ring $S$ (cf. \cite{M}).

Next, let $$G(C, \m) = \ker (G(C) \stackrel{\chi}{\longrightarrow} G(C/\m))$$ (where $\chi$ is the group homomorphism induced by the canonical map $C \to C/\m$) be the congruence subgroup of $G(C) = G(C)^+$ of level $\m.$ Since $\m$ is a nilpotent ideal, it follows that $G(C, \m)$ is nilpotent (see the proof of \cite[Lemma 5.7]{IR}). Moreover, the fact that $G(\mathbb{F}_q)$ coincides with $G(\mathbb{F}_q)^+$ implies that $\sigma$ is surjective and
$$
G(C)/G(C, \m) \simeq G(\mathbb{F}_q).
$$
Now, let $Z$ be the center of $G(\mathbb{F}_q).$ It is well-known that $G(\mathbb{F}_q)/ Z$ is a simple group provided $\vert \mathbb{F}_q \vert \geq 5$ (see \cite[Theorem 5, pg 47]{Stb1}), and it follows from the above remarks that $\tilde{Z} : = \chi^{-1} (Z)$ is a solvable subgroup of $G(C).$

Let $\pi_C \colon \St (\Phi, C) \to G(C)$ be the canonical map, and set $P = \pi^{-1}_C (\tilde{Z}).$ Since
$$
K_2 (\Phi, C) : = \ker (\St (\Phi, C) \stackrel{\pi_C}{\longrightarrow} G(C)^+)
$$
is central in $\St (\Phi, C)$ by \cite[Theorem 2.13]{St2} , we see that $P$ is a solvable subgroup of $\St (\Phi,C)$, and by construction, we have
$$
\St (\Phi, C)/ P \simeq G(\mathbb{F}_q)/Z,
$$
which is simple if $q \geq 5.$

Suppose now that we have a homomorphism $\tilde{\rho} \colon \St (\Phi, C) \to GL_m (K)$ and let $N = \ker \tilde{\rho}.$ If $N \not\subset P,$ then since $\St (\Phi, C) /P$ is simple, it follows that $\St (\Phi, C) = NP.$ Hence $\tilde{\rho} (\St (\Phi, C)) = \tilde{\rho} (P)$ is solvable. On the other hand, $\St (\Phi, C)$ is a perfect group (i.e. coincides with its commutator subgroup --- see \cite[Corollary 4.4]{St1}), so $\tilde{\rho} (\St (\Phi, C))$ is also perfect, and hence trivial.

It remains to consider the case where $N \subset P.$ Let $\mathcal{G} = \tilde{\rho}(\St (\Phi, C))$ and $\mathcal{P} = \tilde{\rho} (P)$ and note that $\mathcal{G}/ \mathcal{P} \simeq \St (\Phi, C)/ P$ is a simple group. By Jordan's Theorem, there exists a normal abelian subgroup $\mathcal{N} \subset \mathcal{G}$ of index dividing $j(m).$ If $\mathcal{N} \not\subset \mathcal{P},$ then $\mathcal{G} = \mathcal{N} \mathcal{P},$ and consequently the quotient $\mathcal{G}/ \mathcal{P}$ is abelian, a contradiction. On the other hand, if $\mathcal{N} \subset \mathcal{P}$, then there exists a sufficiently large integer $M$, such that if $M \in R^{\times}$, then $\vert \mathcal{G} / \mathcal{P} \vert = \vert G(\mathbb{F}_q) / Z \vert > j(m)$, which contradicts the fact that $\vert \mathcal{G} / \mathcal{P} \vert$ divides $j(m).$ Since each of the root subgroups $e_{\alpha}(\mathbb{F}_q)$ has cardinality $q$, it suffices to have $q > j(m)$; thus, we can take $M = \max(5, (j(m))!).$ Then, if $M \in R^{\times},$ we conclude that $\mathcal{G} = \{ e \}$, and hence $\tilde{\rho}$ is trivial, as needed.

\section{Proof of Theorem 1}\label{S:SS}

We now turn to the proof of Theorem 1. Throughout this section, we let $\Gamma$ be a finitely generated group and $K$ an algebraically closed field of characteristic 0.

Recall that our goal is to prove that if $\Gamma$ is not $SS$-rigid, then $\varkappa(\Gamma, n) = \dim X_n (\Gamma)$ is bounded below by a linear function $f(n) = an + b,$ with $a > 0,$ for all $n \geq 1$. First observe that it suffices to show that for some subgroup $\Delta \subset \Gamma$ of finite index $d$, there exist constants $a_{\Delta}$ and $b_{\Delta}$, with $a_{\Delta} > 0$, such that
\begin{equation}\label{E:SubgroupBound}
\dim X_m (\Delta) \geq a_{\Delta} \cdot m + b_{\Delta}
\end{equation}
for all $m \geq 1.$ Indeed, by Corollary \ref{C:Induction}, for $m \geq 1$, we have
$$
\dim X_{md} (\Gamma) \geq \dim X_m (\Delta) \geq a_{\Delta} \cdot m + b_{\Delta} = a_{\Gamma} (md) + b_{\Delta},
$$
where $a_{\Gamma} = \frac{a_{\Delta}}{d}.$ So, setting
$$
f(n) = a_{\Gamma} \cdot n + b_{\Gamma},
$$
with $b_{\Gamma} = b_{\Delta} - a_{\Gamma} (d-1),$ we see from Lemma \ref{L:CharDim} that
\begin{equation}\label{E:SubgroupBound1}
\dim X_n (\Gamma) \geq f(n)
\end{equation}
for all $n \geq m$. If necessary, one can then further adjust $b_{\Gamma}$ such that (\ref{E:SubgroupBound1}) holds for \emph{all} $n \geq 1$.



Now, if $\Gamma$ does not satisfy condition (FAb) (see \S \ref{S:Preliminaries}), there exists a finite-index subgroup $\Delta \subset \Gamma$ that admits an epimorphism $\Delta \twoheadrightarrow \Z.$ Clearly, $\dim X_n (\Delta) \geq n,$ so (\ref{E:SubgroupBound}) holds with $a_{\Delta} = 1$ and $b_{\Delta} = 0.$
Thus, we may, and we will, assume for the remainder of this section that $\Gamma$ satisfies (FAb).

Our first step will be to establish Lemma \ref{L:Simple} below. For the statement, we will need the following
notations. Given an algebraic subgroup $\G \subset GL_n (K)$ and any finite-index subgroup $\Delta \subset \Gamma$,
we let $R(\Delta, \G)$ be the variety of representations $\rho \colon \Delta \to \G$; furthermore, we set
$$
R'(\Delta, \G) = \{ \rho \colon \Delta \to \G \mid \overline{\rho(\Delta)} = \G \},
$$
where, as usual, the bar denotes the Zariski closure.
If the connected component $\G^{\circ}$ is semisimple, then one can show that $R' (\Delta, \G)$ is an open subvariety of $R(\Delta, \G)$ (see \cite[\S 4]{IR1}).

\begin{lemma}\label{L:Simple}
There exists a finite-index subgroup $\Delta \subset \Gamma$ and a simple algebraic group $G$ such that if $\theta_G \colon R(\Delta, G) \to R(\Delta, G)/G$ is the quotient morphism (by the adjoint action), then $\overline{\theta_G (R'(\Delta, G))}$ has positive dimension.
\end{lemma}
\begin{proof}
Let $R_n (\Gamma)_{ss}$ be the set of completely reducible representations $\rho \colon \Gamma \to GL_n (K)$. Then by  Proposition \ref{D:P-KG}, there exists a finite collection $G_1, \dots, G_d$ of algebraic subgroups of $GL_n (K)$ such that
$$
R_n (\Gamma)_{ss} = \bigcup_{\substack{ i \in \{1, \dots, d \}, \\ g \in GL_n(K)}} g R'(\Gamma, G_i) g^{-1}.
$$
Therefore, letting $\pi \colon R_n (\Gamma) \to X_n (\Gamma)$ be the canonical map, we obtain
\begin{equation}\label{E:D-SS1}
X_n (\Gamma) = \bigcup_{i=1}^d \pi (R' (\Gamma, G_i)).
\end{equation}
Now, our assumption that $\Gamma$ is \emph{not} $SS$-rigid means that $\dim X_{n_0} (\Gamma) > 0$ for some $n_0 \geq 1.$ So, it follows from (\ref{E:D-SS1}) that there exists an algebraic subgroup $H \subset GL_{n_0}(K)$ with semisimple connected component $H^{\circ}$ such that $\dim \overline{\pi_{n_0} (R' (\Gamma, H))} > 0$.
Suppose $[H : H^{\circ}] = d$ and let $\Delta$ be the intersection of all subgroups of $\Gamma$ of index $\leq d.$ Since $\Gamma$ is finitely generated, it is clear that $[\Gamma : \Delta] < \infty$, and for any representation $\rho \colon \Gamma \to H,$ we have $\rho (\Delta) \subset H^{\circ}.$
Next, using the commutative diagram
$$
\xymatrix{R_{n_0} (\Gamma) \ar[d]_{\pi_{n_0}^{\Gamma}} \ar[r]^{\mathrm{Res}_{\Delta}^{\Gamma}} & R_{n_0}(\Delta) \ar[d]^{\pi_{n_0}^{\Delta}} \\ X_{n_0} (\Gamma) \ar[r]^{\nu} & X_{n_0}(\Delta)}
$$
(where the horizontal maps are induced by restriction --- see (\ref{E:RestrictionRep}) and (\ref{E:RestrictionChar})),
together with Lemma \ref{L:Restriction}, we conclude that $\overline{\pi_{n_0}^{\Delta} (\mathrm{Res}_{\Delta}^{\Gamma} (R' (\Gamma, H)))}$ has positive dimension, and hence so does $\overline{\pi_{n_0}^{\Delta} (R' (\Delta, H^{\circ}))}.$

Now consider the quotient morphism $\theta \colon R(\Delta, H^{\circ}) \to R(\Delta, H^{\circ})/H^{\circ}$, where $H^{\circ}$ acts via the adjoint action. Then it follows from the above remarks that $\dim \overline{\theta (R' (\Delta, H^{\circ}))} > 0.$ Set $\underline{H} = H^{\circ}/ Z(H^{\circ}),$ where $Z(H^{\circ})$ is the center of $H^{\circ}.$ Since $Z(H^{\circ})$ is finite and $\Delta$ has a finite abelianization, we see that the natural map $R(\Delta, H^{\circ}) \to R(\Delta, \underline{H})$ has finite fibers. Consequently, letting $\underline{\theta} \colon R(\Delta, \underline{H}) \to R(\Delta, \underline{H})/\underline{H}$ be the map induced by $\theta,$ we obtain that $\dim \overline{\underline{\theta}(R'(\Delta, \underline{H}))} > 0.$ But, by \cite[Proposition 14.10]{Bo}, $\underline{H}$ is isomorphic to a product $H_1 \times \cdots \times H_t$, where each $H_i$ is a simple algebraic group, and hence
$$
R(\Delta, \underline{H})/\underline{H} = R(\Delta, H_1)/H_1 \times \cdots \times R(\Delta, H_t)/H_t.
$$
Consequently, for some $i \leq t$, the simple group $G = H_i$ satisfies the required condition.
\end{proof}

Next, let $U$ be an irreducible component of $R'(\Delta, G)$ such that $\dim \overline{\theta_G (U)} > 0$. For any $k \geq 1$, define
$$
V = \{ (\rho_1, \dots, \rho_k) \in U^k \mid \not\exists \ i, j \ \ \text{such that} \ \ \rho_j = \sigma \circ \rho_i \ \ \text{for some} \ \sigma \in {\rm Aut}~G \}.
$$
Let $\theta_{G^k} \colon R(\Delta, G^k) \to R(\Delta, G^k)/G^k$ be the quotient morphism. Viewing $\rho = (\rho_1, \dots, \rho_k) \in U^k$ as a representation $\rho \colon \Delta \to G^k,$ we can consider the restriction
$$
\theta_{G^k} \colon U^k \to (U/G)^k.
$$
Since the group ${\rm Aut}~G/ {\rm Int}~G$ of outer automorphisms of $G$ is finite (see, e.g. \cite[1.5.6]{T}), we see that $\theta_{G^k}(V)$ is a nonempty (hence dense) open subset of $(U/G)^k$, and therefore
\begin{equation}\label{E:DimV}
\dim \overline{\theta_{G^k} (V)} \geq  k \cdot \dim (U/G) \geq k,
\end{equation}
To complete the argument we will need the following lemma.

\begin{lemma}\label{L:DenseImage}{\rm (\cite{AR1}, Lemma 6)}
Suppose for some $\rho_1, \dots, \rho_k \in R' (\Delta, G),$ the image of the diagonal representation
$$
\rho = (\rho_1, \dots, \rho_k) \colon \Delta \to G^k
$$
is not dense. Then there exist $i \neq j$ and an automorphism $\sigma \in {\rm Aut}~G$ such that $\sigma \circ \rho_i = \rho_j.$
\end{lemma}

Let us now fix a matrix realization $G \subset GL_{n_1} (K).$ For $k \geq 1$ and $m = kn_1$, let
$$\pi_m \colon R_m (\Delta) \to X_m (\Delta)$$ denote the canonical map. Given $\rho \in V$, we will view it as a representation $\rho \colon \Delta \to GL_m (K)$ via the diagonal embedding $G \times \cdots \times G \subset GL_m (K).$ Suppose now that $\rho, \rho' \in V$ are such that $\pi_m (\rho) = \pi_m (\rho').$ Clearly, $\rho$ and $\rho'$ are completely reducible, so there exists $g \in GL_{m}(K)$ such that $\rho = g \rho' g^{-1}.$ Since $\overline{\rho(\Delta)} = \overline{\rho'(\Delta)} = G^k$ by Lemma \ref{L:DenseImage}, it follows that 
$g$ normalizes $G^k$, and therefore ${\rm Ad} \ g$ induces an automorphism $\sigma$ of $G^k$. Again, since
the group ${\rm Aut}~G^k/{\rm Int}~G^k$ of outer automorphisms of $G^k$ is finite,
for any $\rho \in V$, the set
$$
T(\rho) = \{\rho' \in V \mid \pi_{m} (\rho') = \pi_{m} (\rho) \}
$$
consists of finitely many orbits under ${\rm Ad}~G^k.$ Consequently, the natural map
$$
\mu \colon R(\Delta, G^k)/G^k \to X_m (\Delta)
$$
has finite fibers on the open subset $\theta_{G^k} (V).$ So, from (\ref{E:DimV}), we obtain
$$
\dim X_m (\Delta) \geq k = \frac{1}{n_1} \cdot m.
$$
Now Lemma \ref{L:CharDim} implies that
$$
\dim X_{m'} (\Delta) \geq a_{\Delta} \cdot m' + b_{\Delta},
$$
with $a_{\Delta} = \frac{1}{n_1}$ and $b_{\Delta} = -\frac{n_1 - 1}{n_1}$, for all $m' \geq n_1.$ Furthermore, if necessary, we can adjust $b_{\Delta}$ so that the inequality actually holds for all $m' \geq 1.$
In view of the remarks made at the beginning of the section, this completes the proof of Theorem 1.

\section{A ``nonlinear" proof to Theorem 2}\label{S:LG}

In this section, we give a proof of Theorem 2 that is based on Theorem \ref{T:Rigidity}. As we already mentioned in the introduction, our original approach, which appeared in \cite{IR1}, was based on estimating the
dimension of the tangent space to $X_n(\Gamma)$ at a point
corresponding to a sufficiently generic representation $\rho \in R_n(\Gamma)$, and then exploiting the
connection, going back to A.~Weil \cite{W}, between this tangent space and the cohomology group
$H^1(\Gamma , \mathrm{Ad} \circ \rho)$. A useful feature of this ``linearized" approach is that it easily
allows one to replace $\Gamma$ by a finite-index subgroup $\Delta$ by
using the injectivity of the restriction map $H^1(\Gamma ,
\mathrm{Ad} \circ \rho) \to H^1(\Delta , \mathrm{Ad} \circ (\rho
\vert \Delta))$. In order to implement this strategy, we relied on our rigidity result from \cite{IR}, which tells us that $\rho$ (as well as some related representations) has a \emph{standard description} on a suitable finite index subgroup $\Delta \subset \Gamma$ (see \cite[Main Theorem]{IR} for the precise statement). In general, this $\Delta$ depends on the representation $\rho$. Now, the refined form of rigidity that we established in Theorem \ref{T:Rigidity} allows us to choose such a subgroup $\Delta \subset \Gamma$ that works uniformly for \emph{all} completely reducible representations of a fixed dimension. Using this, we will give an alternative proof of Theorem 2 that is global in nature and does not require the ``linearization" of the problem.


We begin by fixing notations.
Let $\Phi$ be a reduced irreducible root system of rank $\geq 2$, $G$ the corresponding universal Chevalley-Demazure group scheme, and $K$ an algebraically closed field of characteristic 0. Suppose $R$ is a finitely generated commutative ring such that $(\Phi, R)$ is a nice pair, and let $\Gamma = G(R)^+$ be the elementary subgroup of $G(R).$
Recall that our goal is to show that there exists a constant $c = c(R)$, depending only on $R$, such that
$$
\dim X_n (\Gamma) = \varkappa(\Gamma, n) \leq c \cdot n
$$
for all $n \geq 1.$

The first ingredient in the proof is the following (elementary) lemma.

\begin{lemma}\label{L:RatPoints}
Let $R$ be a finitely generated commutative ring, and fix a finite system of generators $t_1, \dots, t_s$ (as a $\Z$-algebra). Denote by $\Hom (R, K)$ the set of ring homomorphisms $R \to K.$ Then the map
$$
\varphi \colon \Hom (R, K) \to \mathbb{A}_K^s, \ \ \ f \mapsto (f(t_1), \dots, f(t_s))
$$
sets up a bijection between $\mathrm{Hom}(R , K)$ and the set of
$K$-points $U(K)$ of a $\Q$-defined closed subvariety $U \subset
\mathbb{A}_K^s$.

\end{lemma}
\begin{proof}
Since all of the assertions are essentially well-known statements, we only give a brief argument for the sake of completeness. Let $R_K = R \otimes_{\Z} K.$ First notice that we have a natural identification of
$\Hom (R, K)$ with the set of $K$-algebra homomorphisms $\Hom_{K-{\rm alg}} (R_K, K)$, which is obtained by sending $f \in \Hom (R, K)$ to $f \otimes {\rm id}_K$, where ${\rm id}_K \colon K \to K$ is the identity map.
Now, $R_K$ is generated by $t_1 \otimes 1, \dots, t_s \otimes 1$ as a $K$-algebra, so we have a surjective $K$-algebra homomorphism
$$
\psi \colon K[x_1, \dots, x_s] \to R_K, \ \ \ x_i \mapsto t_i \otimes 1, \ i = 1, \dots, s,
$$
where $x_1, \dots, x_s$ are independent variables.
Let $I = \ker \psi$ and set $U = V(I) \subset \mathbb{A}_K^s$ to be the algebraic set determined by $I$. Clearly, a $K$-algebra homomorphism $f \colon K[x_1, \ldots , x_s] \to K$ with $f(x_i) = a_i$ factors
through $\psi$ if and only if $(a_1, \dots , a_s) \in U(K)$. This gives a bijection between $\mathrm{Hom}_{K-\mathrm{alg}}(R_K , K)$
and $U(K)$, and leads to the required bijection $\varphi$.



\end{proof}

\vskip3mm

\addtocounter{thm}{1}

\noindent {\bf Remark 5.2.} It is clear from our construction that the affine varieties $U$ arising from different
generating systems of $R$ are in fact biregularly isomorphic, and in particular, $\dim U$ depends only
on $R$. Furthermore, if $R$ is an integral domain of characteristic zero and $L$ is its field of fractions
then $\dim U = \mathrm{tr. deg.}_{\mathbb{Q}} L$.


\vskip3mm

Next, recall that the points of $X_n (\Gamma)$ correspond to the isomorphism classes of completely reducible representations $\rho \colon \Gamma \to GL_n (K)$ (cf. \cite[Theorem 1.28]{LM}). On the other hand, according to Theorem \ref{T:Rigidity}, for a fixed $n$, there exists a finite index subgroup $\Delta \subset \Gamma$ such that for any (nontrivial) completely reducible representation $\rho \colon \Gamma \to GL_n (K)$, we have
\begin{equation}\label{E:DeltaStandard}
\rho \vert_{\Delta} = (\sigma \circ F) \vert_{\Delta}
\end{equation}
where $F \colon \Gamma \to G(K) \times \cdots \times G(K)$ is a group homomorphism arising from a ring a homomorphism
$$
f \colon R \to \underbrace{K \times \cdots \times K}_{r \ \text{copies}}
$$
(with $r \leq n$)
and $\sigma \colon G(K) \times \cdots \times G(K) \to GL_n (K)$ is a morphism of algebraic groups (in fact, an isogeny --- see \cite[Remark 4.3]{IR1}).

Using this description, we will now parametrize $X_n(\Gamma)$ using
the products $$U^{(r)} = \underbrace{U \times \cdots \times U}_{r \
\mathrm{copies}},$$ with $1 \leq r \leq n$. For $u = (u_1, \ldots , u_r) \in
U^{(r)}(K)$, we let $f^{(r)}_u$ denote the ring homomorphism
$$
(\varphi^{-1}(u_1), \ldots , \varphi^{-1}(u_r)) \colon R
\longrightarrow \underbrace{K \times \cdots \times K}_{r \
\mathrm{copies}}
$$
where $\varphi \colon \mathrm{Hom}(R , K) \to U(K)$ is the bijection
from Lemma \ref{L:RatPoints}. Furthermore, let
$$
F^{(r)}_u \colon \Gamma \longrightarrow G^{(r)}(K) :=
\underbrace{G(K) \times \cdots \times G(K)}_{r \ \mathrm{copies}}
$$
denote the group homomorphism induced by $f^{(r)}_u$.

Now, let us fix some $r \geq 1$ and a morphism of algebraic groups
$\sigma \colon G^{(r)}(K) \to GL_n(K)$, and consider the map
$$
\theta_{r , \sigma} \colon U^{(r)} \to X_n(\Gamma), \ \ u \mapsto
\pi_n^{\Gamma} \circ \sigma \circ F_u,
$$
where $\pi_n^{\Gamma} \colon R_n(\Gamma) \to X_n(\Gamma)$ is the
canonical projection.

\vskip2mm

\begin{lemma}\label{L:Regular}
$\theta_{r , \sigma}$ is a regular map.
\end{lemma}
\begin{proof}
Since $\pi_n^{\Gamma}$ and $\sigma$ are regular maps, it is enough
to show that the map
$$
\lambda^{(r)} \colon U^{(r)} \to R(\Gamma , G^{(r)}), \ \ u \mapsto
F^{(r)}_u,
$$
is regular. This immediately reduces to the case $r = 1$:  namely, what
we need to show is that the map
$$
\lambda \colon U \to R(\Gamma , G), \ \ u \mapsto F_u,
$$
where, for simplicity, we set $f_u = f^{(1)}_u$ and $F_u = F^{(1)}_u$,
is regular. Fix a finite system of generators $\gamma_1, \ldots ,
\gamma_{\ell}$ of $\Gamma$. As $R(\Gamma , G)$ is given the
structure of an affine variety via the embedding
$$
R(\Gamma , G) \hookrightarrow G^{(\ell)}, \ \ \ \ \{\rho \colon \Gamma
\to G\} \mapsto (\rho(\gamma_1), \ldots , \rho(\gamma_{\ell})),
$$
it is enough to show that for any $\gamma \in \Gamma$, the map
$$
\varepsilon_{\gamma} \colon U \to G, \ \ u \mapsto F_u(\gamma),
$$
is regular. To see this, we write $\gamma$ in the form
$$
\gamma = \prod_{i = 1}^d e_{\alpha_i}(r_i)
$$
for some $\alpha_i \in \Phi$ and $r_i \in R$, where for $\alpha \in
\Phi$, we let $e_{\alpha} \colon \mathbb{G}_a \to G$ denote the corresponding
1-parameter root subgroup. Then
$$
F_u(\gamma) = \prod_{i = 1}^d e_{\alpha_i}(f_u(r_i)).
$$
Writing $r_i$ as a polynomial in terms of the chosen finite system
of generators of $R$ (cf. Lemma \ref{L:RatPoints}), we see that the map $u \mapsto
f_u(r_i)$ is regular on $U$, and the fact that
$\varepsilon_{\gamma}$ is regular follows.

\end{proof}

We can now complete the proof of Theorem 2. First, notice that in (\ref{E:DeltaStandard}), there are only finitely many possibilities for the morphism $\sigma$ up to conjugacy. Indeed, since char $K = 0$ and for any $k \geq 1$
$$
G^{(k)} (K) = \underbrace{G(K) \times \cdots \times G(K)}_{k \ \text{copies}}
$$
is a semisimple algebraic group, any finite-dimensional rational representation of $G^{(k)}(K)$ is completely reducible  (cf. \cite[\S 14.3]{H});  furthermore, it is well-known that any irreducible representation of $G^{(k)} (K)$ is a tensor product of irreducible representations of $G(K)$ (see, e.g. \cite[\S 2.5]{Sh}). Now, it follows from the Weyl Dimension Formula (see \cite[Corollary 24.6]{FH}) that $G(K)$ has only finitely many isomorphism classes of irreducible representations of a given dimension. So, we obtain the required finiteness from these remarks, combined with the fact that $r \leq n$ in (\ref{E:DeltaStandard}). Let us fix one representative for $\sigma$ from each conjugacy class, and denote these by $\sigma_1, \dots, \sigma_k$, where
$$
\sigma_i \colon G^{(r_i)} (K) \to GL_n(K)
$$
and $1 \leq r_i \leq n.$ For each $i = 1, \dots, k$, let $\theta_i = \theta_{r_i, \sigma_{i}} \colon U^{(r_i)} \to X_n (\Gamma)$ be the regular map constructed above (notice that $\theta_i$ does not depend on the choice of the representative $\sigma_i$). Also, let $\Delta \subset \Gamma$ be the finite index
subgroup appearing in (\ref{E:DeltaStandard}), and denote by $\nu \colon X_n (\Gamma) \to X_n (\Delta)$ the map induced by the restriction morphism (see (\ref{E:RestrictionChar})).
Then it follows from (\ref{E:DeltaStandard}) that
$$
\nu (X_n (\Gamma)) \subset \bigcup_{i =1}^k \nu(\theta_i (U^{(r_i)})).
$$
Hence, since $\nu$ is an integral morphism by Lemma \ref{L:Restriction}, we obtain
$$
\dim X_n (\Gamma) \leq \max_i \{\dim U^{(r_i)}\} \leq (\dim U) \cdot n,
$$
as needed.


\section{Proof of Theorem 3}\label{S:Aff}

We now turn to the proof of Theorem 3. Even though we stated the result only for complex affine algebraic varieties in the introduction, the argument in fact works over any algebraically closed field of characteristic 0. So, throughout this section, we let $K$ be a fixed algebraically closed field with ${\rm char} \ K = 0.$

Let $S \subset \mathbb{A}_K^t$ be a closed $\Q$-defined affine algebraic variety.  
Denote by $K[S]$ the ring of regular functions on $S$, and let $\Q[S] \subset K[S]$ be the $\Q$-subalgebra of $\Q$-defined regular functions (this is a $\Q$-structure on $K[S]$ --- cf. \cite[AG 11.2]{Bo}). 
Let $r_1, \dots, r_t$ denote the images in $\Q[S]$ of the coordinate functions $x_1, \dots, x_t$ on $\mathbb{A}_{\Q}^t$, and define $R_0 \subset \Q[S]$ to be the $\Z$-subalgebra generated by the $r_1, \dots, r_t.$
Notice that since $S$ is $\Q$-defined, we have  
\begin{equation}\label{E:RationalStruc}
R_0 \otimes_{\Z} K = K[S]. 
\end{equation}



Throughout this section, we take $\Phi = \mathsf{C}_2$ and let $G = Sp_4$ be the corresponding universal Chevalley-Demazure group scheme (see, however, Remark 6.2 below, where we observe that in fact our arguments work for
$\Phi = \mathsf{C}_n$ for any $n \geq 2$).
Furthermore, we let $R = R_0 \left[ \frac{1}{2M} \right],$ where $M$ is the integer appearing in Theorem \ref{T:Rigidity}(ii), and set $\Gamma = G(R)^+.$ In view of (\ref{E:RationalStruc}), we have 
\begin{equation}\label{E:RationalStruc1}
R \otimes_\Z K = K[S].
\end{equation}
Now, given a (nontrivial) completely reducible representation $\rho \colon \Gamma \to GL_4 (K)$, by Theorem \ref{T:Rigidity}(ii), there
exists a ring homomorphism
$$f \colon R \to \underbrace{K \times \cdots \times K}_{k \ \text{copies}}$$
with Zariski-dense image and a morphism of algebraic groups
\begin{equation}\label{E:morphism}
\sigma \colon \underbrace{G(K) \times \cdots \times G(K)}_{k \ \text{copies}} \to GL_4 (K)
\end{equation}
with $k \leq 4$ (in fact, $\sigma$ is an isogeny --- see \cite[Remark 4.3]{IR1})
such that
\begin{equation}\label{E:equationA}
\rho = \sigma \circ F,
\end{equation}
where $F \colon \Gamma \to G^{(k)}(K) := G(K) \times \cdots \times G(K)$ is the group homomorphism induced by $f.$

\begin{lemma}\label{L:StandardRep}
Any nontrivial rational representation $\sigma \colon G(K) \to GL_4(K)$ is equivalent to the standard representation. 
\end{lemma}
\begin{proof}
Since char $K = 0$ and $G(K)$ is an almost simple (in particular, semisimple) algebraic group, any rational representation of $G(K)$ is completely reducible. 
Now, recall that the equivalence class of an irreducible 
representation of $G(K)$ is determined by a unique highest weight (cf. \cite[Theorem 31.3]{H}); denote by $V(\lambda)$ the irreducible representation with highest weight $\lambda.$ It follows from the Weyl Dimension Formula that the nontrivial representations of $G(K)$ of smallest dimension are among the representations of the form $V(\omega),$ where $\omega$ is a fundamental dominant weight (see \cite[Corollary 24.6 and Exercise 24.9]{FH}). In the notations of \cite{Bour1}, the fundamental dominant weights for the root system of type $\mathsf{C}_2$
are $\omega_1 = e_1$, $\omega_2 = e_1 + e_2$, where $\{ e_1, e_2 \}$ is the standard basis of $\R^2.$ Using the Weyl Dimension Formula, one checks directly that
$$
\dim V(\omega_1) = 4 \ \ \ {\rm and} \ \ \ \dim V(\omega_2) = 5.
$$
Moreover, $V(\omega_1)$ is equivalent to the standard representation of $Sp_4(K).$
The lemma now follows.

\end{proof}

Now, since for $k \geq 1$, any rational representation of $G^{(k)}(K)$ is completely reducible (\cite[S 14.3]{H}), and, furthermore, an irreducible representation of $G^{(k)}(K)$ is a tensor product of irreducible representations of $G(K)$ (see \cite[\S 2.5]{Sh}), it follows immediately from the lemma that $k=1$ in (\ref{E:morphism}) and $\sigma$ is conjugate to the standard representation of $Sp_4(K).$

Next, the same argument as in Lemma \ref{L:RatPoints}, in conjunction with (\ref{E:RationalStruc1}), gives a bijection
\begin{equation}\label{E:RatPointsA}
\varphi \colon \Hom (R, K) \to S(K), \ \ \ f \mapsto (f(r_1), \dots, f(r_t)),
\end{equation}
where, as above, $r_1, \dots, r_t$ are the images in $\Q[S]$ of the coordinate functions on $\mathbb{A}_{\Q}^t.$
For a point $s \in S(K)$, let $f_s = \varphi^{-1}(s)$ be the corresponding ring homomorphism and $F_s \colon \Gamma \to Sp_4(K)$ be the group homomorphism induced by $f_s.$ Set $\sigma \colon Sp_4 (K) \to GL_4 (K)$ to be the standard representation. Then by Lemma \ref{L:Regular}, we obtain a regular map
$$
\theta \colon S(K) \to X_4 (\Gamma), \ \ \ s \mapsto \pi_4 \circ \sigma \circ F_s,
$$
where
$\pi_4 \colon R_4(\Gamma) \to X_4 (\Gamma)$ is the canonical projection. Notice that it follows from (\ref{E:equationA}) and Lemma \ref{L:StandardRep} that $\theta(S(K)) = X_4(\Gamma) \setminus \{\rho_0 \}$, where $\rho_0 \colon \Gamma \to GL_4 (\C)$ is the trivial representation. We claim that, furthermore, $\theta$ is injective.
Indeed, we have
$$
Sp_4 (R) = \{ C \in M_4 (R) \mid C T C^t = T \},
$$
where
$$
T = \left( \begin{array}{cccc} 0 & 1 & 0 & 0 \\ -1 & 0 & 0 & 0  \\ 0 & 0 & 0 & 1  \\ 0 & 0 & -1 & 0  \end{array} \right).
$$
For any $r \in R$, consider the following element $\gamma_r \in \Gamma$
$$
\gamma_r = \left( \begin{array}{cccc} -r & 1 & 0 & 0 \\ -1 & 0 & 0 & 0 \\ 0 & 0 & 1 & 0 \\ 0 & 0 & 0 & 1 \end{array} \right).
$$
Suppose now that $f_1, f_2 \colon R \to K$ are two distinct ring homomorphisms, and let $r \in R$ be such that $f_1 (r) \neq f_2 (r).$ Let $F_1, F_2 \colon \Gamma \to Sp_4 (K)$ be the corresponding group homomorphisms. Then we have
$$
(\sigma \circ F_1)(\gamma_r) =   \left( \begin{array}{cccc} -f_1(r) & 1 & 0 & 0\\ -1 & 0 & 0 & 0 \\ 0 & 0 & 1 & 0\\ 0 & 0 & 0 & 1  \end{array} \right) \ \ \ \text{and} \ \ \ (\sigma \circ F_2)(\gamma_r) = \left( \begin{array}{cccc} -f_2(r) & 1 & 0 & 0  \\ -1 & 0 & 0 & 0 \\ 0 & 0 & 1 & 0 \\ 0 & 0 & 0 & 1  \end{array} \right).
$$
Let $\tau_{\gamma_r}$ be the Fricke function corresponding to $\gamma_r$ (see \S \ref{S:Preliminaries}). Then, by our choice of $r$, we have 
$$
\tau_{\gamma_r} (\sigma \circ F_1) = \mathrm{tr}((\sigma \circ F_1)(\gamma_r)) \neq \mathrm{tr} ((\sigma \circ F_2)(\gamma_r)) = \tau_{\gamma_r} (\sigma \circ F_2).$$ 
Since the Fricke functions generate the ring of regular functions on $X_4 (\Gamma)$, it follows that $\pi_4 \circ \sigma \circ F_1$ and $\pi_4  \circ \sigma \circ F_2$ are distinct points of $X_4 (\Gamma).$ Hence, $\theta$ is injective.

To complete the proof, we construct an inverse to $\theta$, as follows. As above, let
$r_1, \dots, r_t \in R$ be the images in $\Q[S]$ of the coordinate functions $x_1, \dots, x_t$ on $\mathbb{A}_{\Q}^t$, and let
$\tau_{\gamma_{r_1}}, \dots, \tau_{\gamma_{r_t}}$ be the Fricke functions corresponding to the elements $\gamma_{r_1}, \dots, \gamma_{r_t} \in \Gamma$. Consider the regular map
$$
\psi \colon X_4 (K) \to \mathbb{A}_{K}^t, \ \ \ [\rho] \mapsto (-\tau_{\gamma_{r_1}} (\rho) + 4, \dots, -\tau_{\gamma_{r_t}} (\rho) + 4),
$$
where $[\rho]$ denotes the equivalence class of a representation $\rho \colon \Gamma \to GL_4 (K)$ (notice that $\psi$ is defined over $\Q$). 
As we have just seen, any completely reducible representation $\rho \colon \Gamma \to GL_4 (K)$ is equivalent to a representation of the form $\sigma \circ F$ for a {\it unique} ring homomorphism $f \colon R \to K.$ So, it follows from (\ref{E:RatPointsA}) and
the explicit description of $(\sigma \circ F)(\gamma_r)$ given above that $\psi (X_4 (K) \setminus \{ [\rho_0] \}) \subset S(K)$. Clearly, we have $\psi \circ \theta = {\rm id}_S$ and $\theta \circ \psi = {\rm id}_{X_4 (K) \setminus \{ [\rho_0] \}}$. Thus, $\theta$ is a $\Q$-defined isomorphism of algebraic varieties, which finishes the proof of Theorem 3.

\vskip3mm

\noindent {\bf Remark 6.2.} We would like to point out that, even though we worked with $Sp_4$ in the proof given above, essentially the same argument goes through for any $Sp_{2n}$ with $n \geq 2.$ Indeed, the crucial ingredient was Lemma \ref{L:StandardRep}, which also holds in the general case. This can be seen as follows. For 
a root system of type $\mathsf{C}_n$, the unique element that takes a fixed Borel subgroup of $G(K) = Sp_{2n}(K)$ to its opposite is $-1$, and hence any rational representation $V$ of $G(K)$ is isomorphic to its dual $V^*$; the latter condition guarantees the existence of a $G(K)$-invariant bilinear form on $V$ (see \cite[\S 31.6]{H}). Suppose now that for $m \leq 2n$, we have an irreducible $m$-dimensional representation
$$
\rho \colon Sp_{2n} (K) \to GL_{m} (K),
$$
and let $V = K^{m}$ be the underlying vector space. Then there exists a $G(K)$-invariant bilinear form $b$ on $V$, which is either symmetric or alternating, so that the image of $\rho$ is contained in an appropriate orthogonal or symplectic group. Now, since $\dim SO_m(K) = \frac{m(m-1)}{2}$
and $\dim Sp_{2 n} (K) = 2 n^2 + n$, it follows from dimension considerations and the fact that $Sp_{2n}(K)$ is an almost simple group that $m = 2n$ and ${\rm Im} \rho = Sp_{2n}(K).$ Consequently, we see that $\rho$ is in fact an automorphism of $Sp_{2n}(K)$; since all automorphisms of the latter are inner, $\rho$ is equivalent to the standard representation of $Sp_{2n}(K)$, as claimed.


\vskip5mm

\bibliographystyle{amsplain}

\end{document}